\documentclass[11pt]{amsart}
\usepackage{amsmath}
\usepackage{tikz}
\usetikzlibrary{matrix,arrows}
\usepackage{amsfonts,amssymb,amsthm, txfonts, pxfonts,amscd}
\usepackage[stretch=10]{microtype}
\usepackage{hyperref}
\def\struckint{\mathop{%
\def\mathpalette##1##2{\mathchoice{##1\displaystyle##2}%
 {##1\textstyle##2}{##1\scriptstyle##2}{##1\scriptscriptstyle##2}}%
\mathpalette
{\vbox\bgroup\baselineskip0pt\lineskiplimit-1000pt\lineskip-1000pt
\halign\bgroup\hfill$}
{##$\hfill\cr{\intop}\cr\diagup\cr\egroup\egroup}%
}\limits}

\def\per{\mathop{\rm per}\nolimits}
\def\det{\mathop{\rm det}\nolimits}

\def\integerpartitions{\mathop{\mathcal{P}_n}\nolimits}

\def\partitionsn{\mathop{\mathcal{P}_{[n]}}\nolimits}
\def\partitionsnk{\mathop{\mathcal{P}_{[n]:k}}\nolimits}

\def\sgn{\mathop{\mathrm{sgn}}\nolimits}
\def\Im{\mathop{\mathrm{Im}}\nolimits}

\def\symmetricn{\mathop{\mathcal{S}_{n}}\nolimits}

\def\part{\mathop{\mbox{Part}}\nolimits}

\def\ell{\mathop{[l]}\nolimits}

\newtheorem{thm}{Theorem}[section]

\newtheorem{prop}[thm]{Proposition}
\newtheorem{cor}[thm]{Corollary}

\newtheorem{conj}[thm]{Conjecture}
\newtheorem{rmk}[thm]{Remark}

\newenvironment{acknowledgment}{\medskip \noindent
        {\bf Acknowledgment.}}{\medskip}

\title{Some algebraic identities for the $\alpha$-permanent}
\author{Harry Crane}\address{Rutgers University\\ Department
of Statistics \\ 
110 Frelinghuysen Road\\
Piscataway, NJ 08854\\
hcrane@stat.rutgers.edu.}
\date{March 3, 2013}
\begin{document}
\maketitle
\begin{abstract}
We show that the permanent of a matrix is a linear combination of determinants of block diagonal matrices which are simple functions of the original matrix.  To prove this, we first show a more general identity involving $\alpha$-permanents: for arbitrary complex numbers $\alpha$ and $\beta$, we show that the $\alpha$-permanent of any matrix can be expressed as a linear combination of $\beta$-permanents of related matrices.  Some other identities for the $\alpha$-permanent of sums and products of matrices are shown, as well as a relationship between the $\alpha$-permanent and general immanants.  We conclude with a discussion of the computational complexity of the $\alpha$-permanent and provide some numerical illustrations.
\end{abstract}

\section{Introduction}\label{section:introduction}
The permanent of an $n\times n$ $\mathbb{C}$-valued matrix $M$ is defined by
\begin{equation}\label{eq:permanent}
\per M:=\sum_{\sigma\in\symmetricn}\prod_{j=1}^n M_{j,\sigma(j)},\end{equation}
where $\symmetricn$ denotes the symmetric group acting on $[n]:=\{1,\ldots,n\}$.  Study of the permanent dates to Binet and Cauchy in the early 1800s \cite{MincPermanent}; and much of the early interest in permanents was in understanding how its resemblance of the determinant,
\begin{equation}\label{eq:determinant}\det M:=\sum_{\sigma\in\symmetricn}\sgn(\sigma)\prod_{j=1}^n M_{j,\sigma(j)},\end{equation}
where $\sgn(\sigma)$ is the parity of $\sigma\in\symmetricn$, reconciled with some stark differences between the two.  An early consideration, answered in the negative by Marcus and Minc \cite{MarcusMinc1961}, was whether there exists a linear transformation $T$ such that $\per TM=\det M$ for any matrix $M$.   In his seminal paper on the $\#$P-complete complexity class, Valiant remarked about the perplexing relationship between the permanent and determinant, 
\begin{quote}
We do not know of any pair of functions, other than the permanent and determinant, for which the explicit algebraic expressions are so similar, and yet the computational complexities are apparently so different. (Valiant  \cite{ValiantPermanent}, p. 189)\end{quote}

In this paper, we hope to give some insight to Valiant's remark, by bringing forth the simple identity
\begin{equation}\label{eq:main identity2}
(-1)^n\per M=\sum_{\pi\in\partitionsn}(-1)^{\downarrow\#\pi}\det(M\cdot\pi),\end{equation}
which expresses the permanent as a linear combination of determinants of block diagonal matrices.   In \eqref{eq:main identity2}, the sum is over the collection $\partitionsn$ of set partitions of $[n]:=\{1,\ldots,n\}$, $\#\pi$ denotes the number of blocks of $\pi\in\partitionsn$, $x^{\downarrow j}:=x(x-1)\cdots(x-j+1)=:(-x)^{\uparrow j}(-1)^j$, and $\det(M\cdot\pi):=\prod_{b\in\pi}\det M[b]$, a product of determinants of the submatrices $M[b]$ with rows and columns labeled by the elements of each block of $\pi$.  Equation \eqref{eq:main identity2} is an immediate corollary of our main identity \eqref{eq:main identity} for the $\alpha$-permanent \cite{VereJones1988} which, for any $\alpha\in\mathbb{C}$, is defined by 
\begin{equation}\label{eq:alpha-permanent}
\per_{\alpha}M:=\sum_{\sigma\in\symmetricn}\alpha^{\#\sigma}\prod_{j=1}^nM_{j,\sigma(j)},\end{equation}
where $\#\sigma$ denotes the number of cycles of $\sigma\in\symmetricn$.  The $\alpha$-permanent generalizes both the permanent and the determinant: $\per M=\per_1 M$ and $\det M=\per_{-1}(-M)$; and, when $M=J_n$, the $n\times n$ matrix of all ones, \eqref{eq:alpha-permanent} coincides with the generating function of the Stirling numbers of the second kind: $\alpha^{\uparrow n}=\sum_{k=0}^n s(n,k)\alpha^{\uparrow k}$, where $s(n,k):=\#\{\pi\in\partitionsn:\#\pi=k\}$.
 Our main identity,
\begin{equation}\label{eq:main identity}
\per_{\alpha\beta}M=\sum_{\pi\in\partitionsn}\beta^{\downarrow\#\pi}\per_{\alpha}(M\cdot\pi)\quad\mbox{for all }\alpha,\beta\in\mathbb{C},\end{equation}
expresses the $\alpha$-permanent as a linear combination of $\beta$-permanents, for any choice of $\alpha$ and $\beta$.  This identity, and its corollaries, could be insightful to understanding the apparent gap between the computational complexity of \eqref{eq:permanent} and \eqref{eq:determinant}.  We discuss these observations further in section \ref{section:computational complexity}, and state a conjecture about the computational complexity of the $\alpha$-permanent.

In addition to \eqref{eq:main identity}, we show other identities for the $\alpha$-permanent of sums and products of matrices.  We separate these identities into two main theorems, calling the first the {\em Permanent Decomposition Theorem}.
\begin{thm}[Permanent Decomposition Theorem]\label{thm:main theorem}
For any $\alpha,\beta\in\mathbb{C}$ and $M\in\mathbb{C}^{n \times n}$,
\[\per_{\alpha\beta}M=\sum_{\pi\in\partitionsn}\beta^{\downarrow\#\pi}\per_{\alpha}(M\cdot\pi),\]
where $\partitionsn$ is the collection of set partitions of $[n]:=\{1,\ldots, n\}$, $\beta^{\downarrow j}:=\beta(\beta-1)\cdots(\beta-j+1)$ and $\per_{\alpha}(M\cdot\pi)=\prod_{b\in\pi}\per_{\alpha}M[b]$, with $M[b]$ denoting the submatrix of $M$ with rows and columns labeled by the elements of $b\subseteq[n]$.
\end{thm}
\begin{thm}\label{thm:main theorem2}
For any $\alpha,\beta\in\mathbb{C}$ and $A,B\in\mathbb{C}^{n\times n}$,
	\begin{equation}\label{eq:permanent sum identity} \per_{\alpha}(A+B)=\sum_{b\subseteq[n]}\per_{\alpha}(AI_b+BI_{b^c}),\end{equation}
	where $I_b$ is a diagonal matrix with $(i,i)$ entry $1$ if $i\in b$ and 0 otherwise and $b^c$ is the complement of $b$ in $[n]$, and
	\begin{equation}\label{eq:permanent product identity} \per_{\alpha}(AB)=\sum_{x\in[n]^n}\per_{\alpha}(B_x)\prod_{j=1}^nA_{j,x_j},\end{equation}
	where $B_x$ is the matrix whose $j$th row is the $x_j$th row of $B$ and $[n]^n:=\{(i_1,\ldots,i_n):1\leq i_j\leq n\mbox{ for all }j=1\ldots,n\}$.
	\end{thm}

Compared to the permanent, the $\alpha$-permanent has been scarcely studied in the literature.  At first glance, the $\alpha$-permanent may not appear as mathematically natural as the permanent or determinant: the $\alpha$-permanent is not an immanant; and it is not clear what, if any, interpretation is possible for values of $\alpha$ other than $\pm1$.  On the other hand, the $\alpha$-permanent arises naturally in statistical modeling of bosons \cite{HoughDetSurvey,Sos2000} (called {\em permanental processes}), as well as in connection to some well-known models in population genetics \cite{Crane2013b}.  In statistical physics applications, permanental processes are complementary to determinantal processes, which model fermions in quantum mechanics.  The Pauli exclusion principle asserts that identical fermions cannot simultaneously occupy the same quantum state, which is reflected in the property that $\det M=0$ if two rows of $M$ are identical.   Just as the exclusion principle does not apply to bosons, $\per M$ need not be zero if $M$ has identical rows.  

Because of the practical potential of permanents, devising efficient methods (random and deterministic) for approximation is a priority.  Recently, there has been some progress in this direction \cite{JerrumSinclairVigoda,LinialSW1998,McCullagh2012}; however,  a provably accurate method which is also practical for large matrices is not yet available.

In section \ref{section:identities}, we prove Theorems \ref{thm:main theorem} and \ref{thm:main theorem2} and observe immediate corollaries.   In section \ref{section:immanants}, we also discuss the relationship between the $\alpha$-permanent and general immanants.  In section \ref{section:computing}, we discuss computation of permanents in three contexts: in section \ref{section:special matrices}, we discuss exact computation of the permanent for some specially structured matrices; in section \ref{section:numerical}, we briefly discuss some approximations of the $\alpha$-permanent based on \eqref{eq:main identity}; in section \ref{section:computational complexity}, we use numerical approximations, see Table \ref{table:1}, and inspection of \eqref{eq:main identity} to make a conjecture about the computational complexity of the $\alpha$-permanent, which has not been studied.   

\section{Identities for the $\alpha$-permanent}\label{section:identities}
In addition to \eqref{eq:main identity}, \eqref{eq:permanent sum identity} and \eqref{eq:permanent product identity}, we observe several immediate corollaries for the $\alpha$-permanent.  We also discuss the relationship between the $\alpha$-permanent and the immanant in section \ref{section:immanants}.  In section \ref{section:proof1}, we prove Theorem \ref{thm:main theorem}; in section \ref{section:proof2}, we prove Theorem \ref{thm:main theorem2}.
\subsection{Proof of Theorem \ref{thm:main theorem}}\label{section:proof1}
Given $\pi\in\partitionsn$, a set partition of $[n]$, we define the $n\times n$ matrix $(\pi_{ij},\,1\leq i,j\leq n)$, also denoted $\pi$, by
\[\pi_{ij}:=\left\{\begin{array}{cc}
1,&\, i\mbox{ and }j\mbox{ are in the same block of }\pi,\\
0,&\,\mbox{otherwise.}\end{array}\right.\]
In this way, the entries of the Hadamard product $M\cdot\pi$ coincide with the entries of $M$, unless a corresponding entry of $\pi$ is $0$.  Since any $\pi$, regarded as a matrix, is the image under conjugation by $\sigma\in\symmetricn$ of a block diagonal matrix, and the $\alpha$-permanent is invariant under conjugation by a permutation matrix, we can regard $M\cdot\pi$ as block diagonal.  If we call any $M\cdot\pi$ a {\em block diagonal projection} of $M$, the Permanent Decomposition Theorem (Theorem \ref{thm:main theorem}) states that the $\alpha\beta$-permanent of a matrix is a linear combination of $\alpha$-permanents of all its  block diagonal projections.  Also, it should be clear from \eqref{eq:alpha-permanent} that
\[\per_{\alpha}(M\cdot\pi)=\prod_{b\in\pi}\per_{\alpha}M[b],\]
where the product is over the blocks of $\pi$ and $M[b]$ denotes the sub-matrix of $M$ with rows and columns indexed by the elements of $b\subseteq[n]$.  Because the diagonal product $\prod_{j=1}^{n}M_{j,\sigma(j)}\pi_{j,\sigma(j)}=0$ for any $\sigma$ whose cycles are not a refinement of $\pi$, we have
\[\per_{\alpha}(M\cdot\pi)=\sum_{\sigma\leq\pi}\alpha^{\#\sigma}\prod_{j=1}^nM_{j,\sigma(j)},\]
where, for $\sigma\in\symmetricn$, we write $\sigma\leq\pi$ to denote that each cycle of $\sigma$ (as a subset of $[n]$) is a subset of some block of $\pi$; in other words, $\sigma\leq\pi$ if and only if the partition of $[n]$ induced by $\sigma$ is a refinement of $\pi$.

To prove \eqref{eq:main identity}, we begin with the right-hand side: let $\alpha,\beta\in\mathbb{C}$ and $M\in\mathbb{C}^{n\times n}$, then
\begin{eqnarray*}
\sum_{\pi\in\partitionsn}\beta^{\downarrow\#\pi}\per_{\alpha}(M\cdot\pi)&=&\sum_{\pi\in\partitionsn}\beta^{\downarrow\#\pi}\sum_{\sigma\leq\pi}\alpha^{\#\sigma}\prod_{j=1}^n M_{j,\sigma(j)}\\
&=&\sum_{\sigma\in\symmetricn}\alpha^{\#\sigma}\prod_{j=1}^{n}M_{j,\sigma(j)}\sum_{\pi\geq\sigma}\beta^{\downarrow \#\pi}\\
&=&\sum_{\sigma\in\symmetricn}\alpha^{\#\sigma}\prod_{j=1}^n M_{j,\sigma(j)}\sum_{j=1}^{\#\sigma}s(\#\sigma,j)\beta^{\downarrow j}\\
&=&\sum_{\sigma\in\symmetricn}(\alpha\beta)^{\#\sigma}\prod_{j=1}^n M_{j,\sigma(j)}\\
&=&\per_{\alpha\beta}M,
\end{eqnarray*}
where $s(n,k):=\#\{\mbox{partitions of }[n]\mbox{ with exactly }k\mbox{ blocks}\}$ is the $(n,k)$ Stirling number of the second kind, whose generating function is $x^n:=\sum_{k=1}^n s(n,k)x^{\downarrow k}$.  Identity \eqref{eq:main identity2} follows readily by taking $\alpha=\beta=-1$ in \eqref{eq:main identity}.

In the context of point processes, McCullagh and M\o ller \cite{McCullagh2005} show the special case of \eqref{eq:main identity} with $\beta=k\in\mathbb{N}$:
\begin{equation}\label{eq:infinitely divisible permanent}
\per_{k\alpha}M=\sum_{\pi\in\partitionsnk}k^{\downarrow\#\pi}\per_{\alpha}(M\cdot\pi),\end{equation}
where $\partitionsnk$ is the collection of partitions of $[n]$ with $k$ or fewer blocks.  They call this the {\em infinite divisibility property} of the $\alpha$-permanent.  Also, with $M=J_n$, the $n\times n$ matrix of all ones, \eqref{eq:main identity} relates closely to the Ewens sampling formula from population genetics \cite{Crane2013b,Ewens1972}.

More generally, we can take $\alpha=-1$ in \eqref{eq:main identity} and let $\beta\in\mathbb{C}$ be arbitrary, in which case we observe
\begin{equation}\label{eq:per det beta}\per_{\beta}M=(-1)^n\sum_{\pi\in\partitionsn}(-\beta)^{\downarrow\#\pi}\det(M\cdot\pi);\end{equation}
in particular, for $k\in\mathbb{N}$,
\[\per_{-k}M=\sum_{\pi\in\partitionsnk}k^{\downarrow\#\pi}\det(M\cdot\pi),\]
which plays a significant role in section \ref{section:computing}.

Immediately from \eqref{eq:main identity}, we also have
\[(-1)^n\det M=\sum_{\pi\in\partitionsn}(-1/\alpha)^{\downarrow\#\pi}\per_{\alpha}(M\cdot\pi).\]
Previously, Vere-Jones \cite{VereJones1988} obtained an expansion for the determinant as a sum of $\alpha$-permanents, whose form hearkens the MacMahon Master Theorem.

\subsection{Proof of Theorem \ref{thm:main theorem2}}\label{section:proof2}
\begin{proof}[Proof of identity \eqref{eq:permanent sum identity}]
Given $b\subseteq[n]$, we write $I_b$ to denote the $n\times n$ diagonal matrix with $(i,i)$ entry
\[I_b(i,i):=\left\{\begin{array}{cc}
	1, & i\in b\\
	0, & \mbox{otherwise.}
	\end{array}\right.
	\]
	With this notation, $AI_b+BI_{b^c}$ denotes the $n\times n$ matrix whose $i$th row is the $i$th row of $A$ if $i\in b$ or the $i$th row of $B$ otherwise.
	
	We prove \eqref{eq:permanent sum identity} using induction and the cofactor expansion of the $\alpha$-permanent.  For $M\in\mathbb{C}^{n\times n}$ and $1\leq i,j\leq n$, we write $M^{(i,j)}$ to denote the $(n-1)\times(n-1)$ matrix obtained from $M$ by removing its $i$th row and $j$th column.  For any $1\leq i\leq n$, the cofactor expansion of $\per_{\alpha}M$ along the $i$th row is
	\[\per_{\alpha}M=\alpha M_{i,i}\per_{\alpha}M^{(i,i)}+\sum_{j\neq i}M_{i,j}\per_{\alpha}M^{(i,j)}.\]
	For $n=1$, \eqref{eq:permanent sum identity} clearly holds and we assume it holds for $n>1$.
	
	For convenience, we write $M:=A+B$ and, since it is understood that we are expanding along the $(n+1)$st row of $M$, we write $M^{j}:=M^{(n+1,j)}$, and likewise for $A$ and $B$.  We also use the shorthand $[A,B]_{a,b}:=AI_a+BI_b$, with $[A,B]_{a}:=[A,B]_{a,a^c}$, when $b=a^c$ is the complement of $a$ in $[n+1]$.
By induction, we have
\begin{eqnarray*}
\per_{\alpha}(A+B)&=&\alpha M_{n+1,n+1}\per_{\alpha}M^{n+1}+\sum_{j=1}^n M_{n+1,j}\per_{\alpha}M^{j}\\
&=&\alpha M_{n+1,n+1}\sum_{b\subseteq[n]}\per_{\alpha}[A^{n+1},B^{n+1}]_{b,b^c\cap[n]}+\sum_{j=1}^n M_{n+1,j}\sum_{b\subseteq[n]\backslash\{j\}}\per_{\alpha}[A^{j},B^{j}]_{b,b^c\backslash\{j\}}\\
&=&\alpha A_{n+1,n+1}\sum_{b\subseteq[n]}\per_{\alpha}[A^{n+1},B^{n+1}]_{b,b^c\cap[n]}+\sum_{j=1}^n A_{n+1,j}\sum_{b\subseteq[n]\backslash\{j\}}\per_{\alpha}[A^{j},B^{j}]_{b,b^c\backslash\{j\}}+\\
&&\quad\quad\quad+\alpha B_{n+1,n+1}\sum_{b\subseteq[n]}\per_{\alpha}[A^{n+1},B^{n+1}]_{b,b^c\cap[n]}+\sum_{j=1}^n B_{n+1,j}\sum_{b\subseteq[n]\backslash\{j\}}\per_{\alpha}[A^{j},B^{j}]_{b,b^c\backslash\{j\}}\\
&=&\sum_{b\subseteq[n+1]:n+1\in b}\per_{\alpha}[A,B]_{b}+\sum_{b\subseteq[n+1]:n+1\notin b}\per_{\alpha}[A,B]_{b}\\
&=&\sum_{b\subseteq[n+1]}\per_{\alpha}[A,B]_b.
\end{eqnarray*}
\end{proof}
\begin{cor}
For $A\in\mathbb{C}^{n\times n}$ and $\alpha\in\mathbb{C}$,
\[\per_{\alpha}(A+I_n)=\sum_{b\subseteq[n]}\alpha^{n-\#b}\per_{\alpha}A[b],\]
where $I_n$ is the $n\times n$ identity matrix.  In particular,
\begin{equation}\label{eq:determinant sum identity}\det(A+I_n)=\sum_{b\subseteq[n]}\det A[b]\end{equation}
and
\[\per(A+I_n)=\sum_{b\subseteq[n]}\per A[b].\]
\end{cor}

\begin{rmk}
A similar identity to \eqref{eq:permanent sum identity} is given in Theorem 1.4 (chapter 2) of \cite{MincPermanent}, but, even in the case $\alpha=1$, that identity differs slightly.  We also know, through personal communications, that \eqref{eq:determinant sum identity} is known within the context of determinantal processes, and presumably more widely.  Regardless, we do not formally know of any identities that resemble \eqref{eq:permanent sum identity}, even in special cases.
\end{rmk}

\begin{proof}[Proof of identity \eqref{eq:permanent product identity}]
To prove \eqref{eq:permanent product identity}, we use the cofactor expansion of the $\alpha$-permanent together with induction, as in the proof of \eqref{eq:permanent sum identity}.  Let $A,B\in\mathbb{C}^{n\times n}$ and $\alpha\in\mathbb{C}$.  For $n=1$, \eqref{eq:permanent product identity} is trivially true.  In the induction step, we assume \eqref{eq:permanent product identity} holds for $n>1$.

For $x=(x_1,\ldots,x_n)\in[n]^n$, we write $B_x$ to denote the $n\times n$ matrix whose $(i,j)$ entry is $B_x(i,j):=B(x_i,j)$, i.e.\ the $i$th row of $B_x$ is the $x_i$th row of $B$.  We write $M:=AB$ and, again, $M^j:=M^{(n+1,j)}$ denotes the submatrix of $M$ obtained by removing the $n+1$st row and $j$th column.  We have the following for $A,B\in\mathbb{C}^{(n+1)\times(n+1)}$.
\begin{eqnarray*}
\per_{\alpha}(AB)&=&\alpha M_{n+1,n+1}\per_{\alpha} M^{n+1}+\sum_{j=1}^n M_{n+1,j}\per_{\alpha}M^j\\
&=&\alpha M_{n+1,n+1}\sum_{x\in[n]^n}\per_{\alpha}B^{n+1}_x\prod_{k=1}^n A_{k,x_k}+\sum_{j=1}^nM_{n+1,j}\sum_{x\in[n]^n}\per_{\alpha}B^{n+1}_x\prod_{k=1}^nA_{k,x_k}\\
&=&\alpha\left(\sum_{i=1}^{n+1}A_{n+1,i}B_{i,n+1}\right)\sum_{x\in[n]^n}\per_{\alpha}B^{n+1}_x\prod_{k=1}^nA_{k,x_k}+\sum_{j=1}^{n}\left(\sum_{i=1}^{n+1}A_{n+1,i}B_{i,j}\right) \sum_{x\in[n]^n}\per_{\alpha}B^{n+1}_x\prod_{k=1}^nA_{k,x_k}\\
&=&\alpha\left(\sum_{i=1}^{n+1}A_{n+1,i}B_{i,n+1}\right)\sum_{x\in[n]^n}\per_{\alpha}B^{n+1}_x\prod_{k=1}^nA_{k,x_k}+ \sum_{i=1}^{n+1}A_{n+1,i}\sum_{j=1}^{n}B_{i,j}\sum_{x\in[n]^n} \per_{\alpha}B^{n+1}_x\prod_{k=1}^nA_{k,x_k}\\
&=&\sum_{i=1}^{n+1}A_{n+1,i}\sum_{j=1}^{n+1}B_{i,j}(1+(\alpha-1)\delta_{j,n+1})\sum_{x\in[n]^n}\per_{\alpha}B^{n+1}_x\prod_{k=1}^nA_{k,x_k}\\
&=&\sum_{i=1}^{n+1}\sum_{x\in[n+1]^{n+1}:x_{n+1}=i}\per_{\alpha}B_x\prod_{k=1}^{n+1}A_{k,x_k}\\
&=&\sum_{x\in[n+1]^{n+1}}\per_{\alpha}B_x\prod_{k=1}^{n+1}A_{k,x_k}.
\end{eqnarray*}
This completes the proof.
\end{proof}
\begin{rmk}
Theorem 1.3 (chapter 2) in Minc \cite{MincPermanent} restates a theorem of Binet and Cauchy for the permanent of a product of matrices, but that identity is different than \eqref{eq:permanent product identity}.  Both \eqref{eq:permanent sum identity} and \eqref{eq:permanent product identity} are more suitable to statistical applications than their counterparts in \cite{MincPermanent}, because they give the explicit normalizing constant of a probability distribution on subsets of $[n]$, which is relevant to various applications involving clustering.
\end{rmk}

\subsection{Immanants and the $\alpha$-permanent}\label{section:immanants}
In representation theory, the permanent corresponds to the {\em immanant} associated to the trivial representation of $\symmetricn$, while the determinant is the immanant corresponding to the alternating representation, the only other one-dimensional representation of $\symmetricn$.  With $\lambda\dashv n$ indicating that $\lambda$ is a partition of the integer $n$, the immanant indexed by any $\lambda\dashv n$ is
\begin{equation}\label{eq:immanant}
\Im_{\lambda} M:=\sum_{\sigma\in\symmetricn}\chi_{\lambda}(\sigma)\prod_{j=1}^{n}M_{j,\sigma(j)},\end{equation}
where $\chi_{\lambda}(\sigma)$ is the character associated to the irreducible representation $S^{\lambda}$, the Specht module corresponding $\lambda\dashv n$; see \cite{FultonHarris}.  Aside from $\alpha=\pm1$, the $\alpha$-permanent \eqref{eq:alpha-permanent} does not correspond to an immanant of a particular index.  However, there is a precise relationship between \eqref{eq:alpha-permanent} and \eqref{eq:immanant}, which we now discuss.
\subsubsection{Characters and immanants}

For each $\lambda\dashv n$, the character $\chi_{\lambda}:\symmetricn\rightarrow\mathbb{R}$ associated to the irreducible representation $\rho_{\lambda}$ is a {\em class function}, i.e.\ $\chi_{\lambda}(\sigma)=\chi_{\lambda}(\tau\sigma\tau^{-1})$ for all $\sigma,\tau\in\symmetricn$.  The conjugacy classes of $\symmetricn$ correspond to $\integerpartitions:=\{\lambda\dashv n\}$, the collection of integer partitions of $n$,
\[\integerpartitions:=\left\{(n_1,\ldots,n_k):\,n_1\geq\cdots\geq n_k>0,\,n_1+\cdots+n_k=n\right\}.\]
  Moreover, the collection $(\chi_{\lambda},\,\lambda\in\integerpartitions)$ is orthnormal and forms a basis for the space of class functions on $\symmetricn$.  The $\alpha$-permanent is a polynomial in $\alpha$ and, though $\alpha^{\#\bullet}$ is not a character, it is a class function.  Hence, for every $\alpha$ there exist constants $(c_{\lambda}^{\alpha},\,\lambda\in\integerpartitions)$ such that
\begin{equation}\label{eq:linear comb imm}
\per_{\alpha}M=\sum_{\lambda\dashv n}c_{\lambda}^{\alpha}\Im_{\lambda}M.\end{equation}
Special cases are $\alpha=\pm1$, for which $c_{\lambda}^{1}=\delta_{\lambda,(n)}$, the Dirac mass at $\lambda=(n)$, and $c_{\lambda}^{-1}=(-1)^n\delta_{\lambda,1^n}$, the Dirac mass at $\lambda=1^n:=(1,\ldots,1)\in\integerpartitions$.  

For $n\in\mathbb{N}$, let $\mathbf{X}:=(X_{\lambda}(\nu),\,\lambda,\nu\dashv n)$ be a $\integerpartitions\times\integerpartitions$ matrix, i.e.\ its rows and columns are labeled by $\integerpartitions$, with $(\lambda,\nu)$ entry equal the value the $\lambda$-character takes on the conjugacy class indexed by $\nu$.  For any $(\lambda,\nu)\in\integerpartitions\times\integerpartitions$, we write $\mathbf{X}^{(\lambda,\nu)}$ to denote the minor determinant of $\mathbf{X}$ with $\lambda$ row and $\nu$ column removed.  We then define $\mathbf{Y}:=(\mathbf{X}^{(\lambda,\nu)},\,\lambda,\nu\dashv n)$ as the matrix of minor determinants.

\begin{thm}
For every $\lambda\in\integerpartitions$, define the function $c_{\lambda}:\mathbb{C}\rightarrow\mathbb{C}$ by
\begin{equation}\label{eq:lambda fct}
c_{\lambda}(\alpha):=\frac{1}{n!}\sum_{\sigma\in\symmetricn}\alpha^{\#\sigma}\chi_{\lambda}(\sigma).\end{equation}
Then the $\alpha$-permanent satisfies
\[\per_{\alpha}M=\sum_{\lambda\in\integerpartitions}c_{\lambda}(\alpha)\Im_{\lambda}M\]
 for all $\alpha\in\mathbb{C}$ and $M\in\mathbb{C}^{n\times n}$.  Alternatively, $\mathbf{c}:=(c_{\lambda}(\alpha),\,\lambda\dashv n)$ solves the equation
\begin{equation}\label{eq:linear system}\mathbf{A}=\mathbf{X}\mathbf{c},\end{equation}
where $\mathbf{X}$ is defined before the theorem, $\mathbf{A}:=(\alpha^{\#\lambda},\,\lambda\dashv n)$ is a $\integerpartitions\times1$ column vector and $\#\lambda$ denotes the number of parts of $\lambda$.  The solution to \eqref{eq:linear system} is
\begin{equation}\label{eq:linear system solution}
\mathbf{c}=\mathbf{Y}\mathbf{A}/\det\mathbf{X},\end{equation}
where $\mathbf{Y}$ is defined above.
\end{thm}
\begin{proof}
That such a function $c_{\lambda}$ exists is a consequence of \eqref{eq:linear comb imm} and the surrounding discussion.  To obtain \eqref{eq:lambda fct}, we note that
\[\per_{\alpha}M=\sum_{\sigma\in\symmetricn}\alpha^{\#\sigma}\prod_{j=1}^n M_{j,\sigma(j)}=\sum_{\sigma\in\symmetricn}\sum_{\lambda\dashv n}c_{\lambda}(\alpha)\chi_{\lambda}(\sigma)\prod_{j=1}^n M_{j,\sigma(j)}\quad\mbox{for all }M\in\mathbb{C}^{n\times n};\]
whence,
\begin{equation}\label{eq:linear comb char}
\alpha^{\#\sigma}=\sum_{\lambda\dashv n}c_{\lambda}(\alpha)\chi_{\lambda}(\sigma),\quad\mbox{for all }\sigma\in\symmetricn.\end{equation}
The Fourier inversion theorem gives \eqref{eq:lambda fct}.

To see \eqref{eq:linear system solution}, it is enough to show that $\mathbf{Y}\mathbf{X}=\det(\mathbf{X})I_{\integerpartitions}$, where $I_{\integerpartitions}$ is the $\integerpartitions\times\integerpartitions$ identity matrix.  Indeed, by orthogonality of characters,
\[(\mathbf{Y}\mathbf{X})_{\lambda\nu}=\sum_{\mu\dashv n}\chi_{\lambda}(\mu)\mathbf{X}^{(\mu,\nu)}=\left\{\begin{array}{cc}\det\mathbf{X},&\lambda=\nu\\
0,& \mbox{otherwise,}\end{array}\right.\]
for all $\lambda,\nu\dashv n$.  Orthogonality of characters, and hence the columns of $\mathbf{X}$, implies $\det\mathbf{X}\neq0$, completing the proof.
\end{proof}
\begin{cor}
For any $\alpha,\beta\in\mathbb{C}$,
\begin{equation}\label{eq:permanent immanant identity}
\per_{\alpha}M=\sum_{\lambda\dashv n}c_{\lambda}^{-\beta}\sum_{\pi\in\partitionsn}(-\alpha/\beta)^{\downarrow\#\pi}\Im_{\lambda}(M\cdot\pi),\end{equation}
where $c_{\lambda}^{-\beta}:=c_{\lambda}(-\beta)$ is defined in \eqref{eq:lambda fct}.
\end{cor}
\begin{proof}
Immediately, by combining \eqref{eq:main identity} and \eqref{eq:linear comb imm}, we obtain the $\alpha$-permanent as an explicit linear combination of immanants, in terms of the coefficients $(c_{\lambda}(\beta),\,\lambda\dashv n)$, for arbitrary $\beta\in\mathbb{C}$, yielding \eqref{eq:permanent immanant identity}.
\end{proof}
From \eqref{eq:permanent immanant identity}, we can write
\[\per_{\beta}(M\cdot \pi)=\sum_{\sigma\leq\pi}\prod_{j=1}^n M_{j,\sigma(j)}\sum_{\lambda\dashv n}c_{\lambda}^{\beta}\chi_{\lambda}(\sigma)\]
For $\sigma\in\symmetricn$ and $\pi\in\partitionsn$, we write $\sigma\sim\pi$ to denote that the cycles of $\sigma$ correspond to the blocks of $\pi$; and, since $\chi_{\lambda}$ is a class function, we define $\chi_{\lambda}(\pi)$ to be the common value that $\chi_{\lambda}$ takes on $\{\sigma\sim\pi\}$.
Defining $f(\pi)=\per_{\beta}(M\cdot\pi)$ and $g(\pi)=\sum_{\sigma\sim\pi}\prod_{j=1}^n M_{j,\sigma(j)}\sum_{\lambda\dashv n}c_{\lambda}^{\beta}\chi_{\lambda}(\pi)$, we obtain \begin{equation}\label{eq:Mobius inversion}
\left[\sum_{\sigma\sim\pi}\prod_{j=1}^n M_{j,\sigma(j)}\right]\sum_{\nu\dashv n}\sum_{\sigma\in\symmetricn}\beta^{\#\sigma}\chi_{\nu}(\pi)\chi_{\nu}(\sigma)=\sum_{\pi'\leq\pi}\per_{\beta}(M\cdot\pi')\prod_{b\in\pi}(-1)^{\#\pi'_{|b}-1}(\#\pi'_{|b}-1)!
\end{equation} as a consequence of the M\"obius inversion formula (\cite{StanleyI}, proposition 3.7.1 and example 3.10.4).  As a special case, when $M=J_n$ in \eqref{eq:Mobius inversion}, 
\[\prod_{b\in\pi}(\#b-1)!\sum_{\nu\dashv n}\sum_{\sigma\in\symmetricn}\beta^{\#\sigma}\chi_{\nu}(\pi)\chi_{\nu}(\sigma)=\sum_{\pi'\leq\pi}\prod_{b\in\pi'}\beta^{\uparrow\#b}\prod_{b\in\pi}(-1)^{\#\pi'_{|b}-1}(\#\pi'_{|b}-1)!.\]

\section{Computing the $\alpha$-permanent}\label{section:computing}
\subsection{Permanents of special matrices}\label{section:special matrices}
Although the $\alpha$-permanent is difficult to compute in general, it can be computed explicitly for some specially structured matrices.  For example, for any $\sigma\in\symmetricn$, let $P_{\sigma}$ be its matrix representation.  Then $\per_{\alpha} P_{\sigma}=\alpha^{\#\sigma}$.  Also, if, as before, we regard $\pi\in\partitionsn$ as a matrix,
then 
\[\per_{\alpha}\pi=\prod_{b\in\pi}\alpha^{\uparrow\#b},\]
the product over the blocks of the rising factorial $\alpha^{\uparrow\#b}:=\alpha(\alpha+1)\cdots(\alpha+\#b-1)$.  Finally, for a $2\times 2$ matrix
\[A:=\begin{pmatrix}
A_{11} & A_{12}\\
A_{21} & A_{22}\end{pmatrix},\]
define $A[n_1,n_2]$ as the $(n_1+n_2)\times(n_1+n_2)$ {\em block matrix} with $(i,j)$ entry
\[A[n_1,n_2](i,j):=\left\{\begin{array}{cc}
A_{11}, & i,j\leq n_1\\
A_{22}, & i,j>n_1\\
A_{12}, & i\leq n_1,\, j>n_1\\
A_{21}, & i>n_1,\,j\leq n_1.\end{array}\right.\]
For example,
\[A[2,3]:=\begin{pmatrix}
A_{11} & A_{11} & A_{12} & A_{12} & A_{12}\\
A_{11} & A_{11} & A_{12} & A_{12} & A_{12}\\
A_{21} & A_{21} & A_{22} & A_{22} & A_{22}\\
A_{21} & A_{21} & A_{22} & A_{22} & A_{22}\\
A_{21} & A_{21} & A_{22} & A_{22} & A_{22}
\end{pmatrix}.\]
Such block matrices arise naturally in statistical modeling, and the following identity was obtained in \cite{RubakMollerMcC2010}.
\begin{prop}[Proposition 1 in A.2 \cite{RubakMollerMcC2010}]
Let $A$ be $2\times 2$ as above and define 
\[\rho:=\frac{A_{11}A_{22}}{A_{12}A_{21}}.\]
Then 
\[\per_{\alpha}A[n_1,n_2]=A_{11}^{n_1}A_{22}^{n_2}\alpha^{\uparrow n_1}\alpha^{\uparrow n_2}\sum_{j=0}^{n_1\wedge n_2}\frac{n_1^{\downarrow j}n_2^{\downarrow j}\rho^j}{j!\alpha^{\uparrow j}},\]
where $\alpha^{\uparrow n_1}\alpha^{\uparrow n_2}/\alpha^{\uparrow j}:=0$ whenever the quantity is undefined.
\end{prop}

Another natural consideration are what we call {\em homogeneously symmetric matrices} $\mathbf{H}[a,b;n]$.  For $a,b\in\mathbb{C}$, we define $\mathbf{H}[a,b;n]$ to be the $n\times n$ matrix with diagonal entries $a$ and off-diagonal entries $b$.  Since $n$ is typically understood, we will usually write $\mathbf{H}[a,b]$.  For example,
\[\mathbf{H}[a,b;3] := \begin{pmatrix} a & b & b \\ b & a & b \\ b & b & a\end{pmatrix}.\]
\begin{prop}
Let $\mathbf{H}[a,b]$ be a homogeneously symmetric matrix and write $d:=b/a$.  Then, for any $\alpha\in\mathbb{C}$,
\begin{equation}\label{eq:per homo sym}
\per_{\alpha}\mathbf{H}[a,b]=a^n\sum_{k=1}^n\sum_{l=0}^n c(n,k,l)\alpha^k d^{n-l},\end{equation}
where $c(n,k,l)$ denotes the number of permutations of $n$ having exactly $k$ cycles and exactly $l$ fixed points.
\end{prop}
\begin{proof}
It is plain that $\mathbf{H}[a,b]=a\mathbf{H}[1,b/a]$ and $\per_{\alpha}\mathbf{H}[a,b]=a^n\per_{\alpha}\mathbf{H}[1,d]$, $d:=b/a$.  So it is enough to consider matrices with unit diagonal.  Identity \eqref{eq:per homo sym} follows directly from the definition of the $\alpha$-permanent and the coefficients $c(n,k,l)$.
\end{proof}
The numbers $c(n,k,l)$ are related to the {\em rencontres numbers} $f(n,l)$, where $f(n,l)$ is the number of permutations of $[n]$ with exactly $l$ fixed points.  The rencontres numbers correspond to the number of {\em partial derangements} of an $n$-set, with a {\em derangement} defined as a permutation with zero fixed points.   By this definition,
\[f(n,l)=\sum_{k}c(n,k,l),\]
prompting us to call $c(n,k,l)$ the $(n,k,l)$-{\em generalized rencontres number}.
The rencontres numbers satisfy $f(0,0)=1,\,f(1,0)=0$ and, for $n>1$,
\[f(n+1,0)=n(f(n,0)+f(n-1,0)).\]
For $0\leq l\leq n$, $f(n,l)$ satisfies
\[f(n,l)={{n}\choose{l}}f(n-l,0).\]
The $(n,k,l)$ Rencontres numbers satisfy 
\begin{equation}\label{eq:(n,k,l) Rencontres recursion}
c(n,k,l)=\left\{\begin{array}{cc}
c(n-1,k-1,l-1)+(n-l-1)c(n-1,k,l)+(l+1)c(n-1,k,l+1),& 0\leq l\leq k\leq n,\,k\leq n-l\\
0,&\mbox{otherwise.}\end{array}\right.\end{equation}
The recursion \eqref{eq:(n,k,l) Rencontres recursion} has a simple combinatorial interpretation.  We call $\sigma\in\symmetricn$ an $(n,k,l)$ permutation if it has exactly $k$ cycles and exactly $l$ fixed points.  An $(n,k,l)$ permutation can be obtained from $\sigma\in\mathcal{S}_{n-1}$ by inserting $n+1$ into one of the cycles of $\sigma$ only if
\begin{itemize}
	\item[(1)] $\sigma$ is an $(n-1,l-1,k-1)$ permutation: in this case, we obtain an $(n,k,l)$ permutation by appending $n+1$ to $\sigma$ as a fixed point;
	\item[(2)] $\sigma$ is an $(n-1,k,l)$ permutation: in this case, we insert $n+1$ into one of the $k-l$ cycles that are not fixed; there are $n-l-1$ possible points of insertion;
	\item[(3)] $\sigma$ is an $(n-1,k,l+1)$ permutation: in this case, we choose from one of the $l+1$ fixed points and insert $n+1$, bringing the number of fixed points in the new permutation to $l$ while keeping the total number of cycles fixed.
\end{itemize}
The recursion \eqref{eq:(n,k,l) Rencontres recursion} makes these numbers easy to compute.  The $(n,k,l)$-generalized rencontres numbers are also related to the number of derangements.  Let $g(n,k)$ denote the number of permutations of $[n]$ with $k$ cycles and zero fixed points.  Clearly, $g(n,k)=c(n,k,0)$ and $f(n,0)=\sum_{k}g(n,k)$.  We also have the expression
\begin{equation}\label{eq:no fixed identity}
c(n,k,l)=\left\{\begin{array}{cc}
{{n}\choose{l}}g(n-l,k-l),& l<n\\
\delta_{kl},& l=n,\end{array}\right.
\end{equation}
where $\delta_{kl}=1$ if $k=l$ and $0$ otherwise.  Appendix \ref{section:rencontres} contains tables of the generalized rencontres numbers for $n\leq10$.

\subsection{Approximating the $\alpha$-permanent}\label{section:numerical}
This and the following section concern the approximation and computational complexity of the $\alpha$-permanent. Among previous algorithms and asymptotic approximations of the permanent, the randomized method in \cite{KouMcCullagh2009} seems the most practical, but is not provably good.  Other methods, such as those in \cite{JerrumSinclairVigoda, LinialSW1998}, are either not practical (the method in \cite{JerrumSinclairVigoda} still requires $O(n^7(\log n)^4)$ operations) or not accurate (the approximation in \cite{LinialSW1998} is accurate up to a factor $e^n$).  All of the aforementioned approximation methods are valid for non-negative matrices and, in the case of \cite{KouMcCullagh2009}, $\alpha>0$.  The assumption that all entries of $M$ are non-negative is, in some sense, the permanental analog to positive semi-definiteness in the determinantal case.  For $\alpha>0$ and $M_{ij}\geq0$, each term in \eqref{eq:alpha-permanent} is non-negative.  Analogously, for $M$ positive semi-definite, $\det(M\cdot\pi)\geq0$ for all $\pi\in\partitionsn$ by Sylvester's criterion.  

Positivity of the terms of \eqref{eq:alpha-permanent} is crucial to the approximation in \cite{KouMcCullagh2009}, which uses the statistical method of importance sampling (IS) to approximate $\per_{\alpha}M$.  An introduction to IS is given in \cite{PressIS2007} pp.\ 411--412.  In IS, we estimate the sum
\[\per_{\alpha}M=\sum_{\sigma\in\symmetricn}\alpha^{\#\sigma}\prod_{j=1}^n M_{j,\sigma(j)}\]
by drawing random permutations $\sigma_1,\ldots,\sigma_N$, for $N$ as large as is computationally feasible, and, writing $f(\sigma):=\alpha^{\#\sigma}\prod_{j=1}^n M_{j,\sigma(j)}$, we approximate $\per_{\alpha}M$ by
\[\widehat{\per}_{\alpha}M:=\sum_{i=1}^N\frac{f(\sigma_i)}{P(\sigma_i)},\]
where $P(\sigma)$ is the probability of $\sigma$.  The choice of $\widehat{\per}_{\alpha}M$ is so that our estimate is {\em unbiased} and, by the law of large numbers, if $N$ is large enough, $\widehat{\per}_{\alpha}M$ should be close to the true value $\per_{\alpha}M$.  The efficiency of this method relies on choosing $P$ which is close to optimal.  It is known that $P(\sigma)\propto f(\sigma)$ is the optimal choice, which is rarely practical.  Kou and McCullagh use a variation of importance sampling, called sequential importance sampling, which allows them to draw from $P$ which is nearly optimal.

Because the determinant can be computed efficiently, \eqref{eq:per det beta} is a natural candidate for use in IS.  To use \eqref{eq:per det beta} in IS, we choose a sequence $\pi_1,\ldots,\pi_N$ of random partitions from $P^*$ and estimate $(-1)^n\per_{\alpha}M$ by
\[(-1)^n\widehat{\per}_{\alpha}M=\sum_{j=1}^N\frac{(-\alpha)^{\downarrow\#\pi}\det(M\cdot\pi)}{P^*(\pi)}.\]
 In view of \eqref{eq:per det beta},
the optimal importance sampling distribution is $P^*(\pi)\propto|(-\alpha)^{\downarrow\#\pi}\det(M\cdot\pi)|$.  One issue is that, even if we could draw from $P^*$ and we assume $M$ is positive definite, for any $\alpha\in\mathbb{R}$, there will be both positive and negative terms in \eqref{eq:per det beta}, and our estimate based on importance sampling could have extraordinarily high variance.  In fact, some preliminary numerical tests illustrate this issue, but suggest that when $\alpha$ is a negative integer, the problem is simplified and a reasonable approximation is possible.

\subsubsection{Some numerical illustrations}
This section is a prelude to section \ref{section:computational complexity}, where we draw on observations from numerical tests to analyze, intuitively, any further reaching implications for the complexity of the $\alpha$-permanent.
We show only a few numerical approximations in order to illustrate IS, as well as some of the caveats discussed above.  

The sampling distribution $P^*$ we used is the Pitman-Ewens$(a,\theta)$ distribution \cite{Pitman2005}, for various choices of $(a,\theta)$.  We would not expect that this distribution is near optimal; however, when $a=-k$ for $k\in\mathbb{N}$, we can choose $a<0$ and $\theta=-ka$, which restricts the Pitman-Ewens distribution to $\partitionsnk$.  On the other hand, when $a$ is not a negative integer, it seems there is no optimal choice of $(a,\theta)$; in this case, all estimates, based on different choices of $(a,\theta)$, returned values similar to those in Table \ref{table:1} for the $1$ and $-2.5$ case (the table reflects $a=0$ and $\theta=1$ in these cases, the {\em standard} Ewens sampling formula \cite{Ewens1972}).  The purpose of the demonstration is to show the striking difference when using \eqref{eq:per det beta} and a naive choice of $P^*$ to approximate the $\alpha$-permanent in the cases when $\alpha$ is a negative integer and all other cases.  The summaries are given in Table \ref{table:1}.
\begin{table}
\begin{tabular}{c|ccc}
& Actual & Estimate $\pm$ standard error & relative error\\\hline
$\per_{-2}X_1$ & 407.52 & 406.43 $\pm$ 10.32 & 2.53\%\\
$\per_{-3}X_1$ & 117488 & $117300\pm3957$ & 3.37\%\\
$\per_{-2.5}X_1$ & -44088 & $-427609\pm7.03\times10^6$ & 1643\%\\
$\per X_1$ & $1.6\times10^8$ & $3.1\times10^8\pm4.8\times10^9$ & 1548\%\\
\hline
$-2$-permanent & & $199.19\pm8.10$ & 4.01\%\\
$-3$-permanent & & $41692\pm1876$ & 4.45\%\\
$-2.5$-permanent & &$43338\pm995181$ & 2296\%\\
$1$-permanent & & $-1.1\times10^{7}\pm6.3\times10^8$ & 5727\%\\
\end{tabular}
\caption{Estimates of the $\alpha$-permanent for $\alpha=1,-2,-2.5,-3$ for randomly generated, symmetric positive definite matrices.  The first four rows are the $\alpha$-permanent for $X_1$, given in the appendix.  Each of the bottom four rows reflects 200 iterations of generating random, symmetric positive definite matrices ($8\times 8$) and comparing, for each iteration, the estimated quantity and the actual quantity.}
\label{table:1}
\end{table}

The standard errors for estimating the ordinary ($\alpha=1$) and $-2.5$-permanents exceed the value of the quantity being estimated.  In these cases, our approximation is unstable because it involves a sum of both large positive and large negative numbers, a consequence of the coefficient $(-\beta)^{\downarrow\#\pi}$ in \eqref{eq:per det beta}, which is alternating between positive and negative values for $\#\pi>\alpha$.  Stability is only achieved for negative integer values.

\subsection{Computational complexity of the $\alpha$-permanent}\label{section:computational complexity}
We conclude with some brief remarks about computational complexity of the $\alpha$-permanent, which, to our knowledge, has not been studied.  For any $\alpha\in\mathbb{R}$, we have
\[\per_{\alpha}M=(-1)^n\sum_{\pi\in\partitionsn}(-\alpha)^{\downarrow\#\pi}\det(M\cdot\pi).\]
It seems intuitive that $\per_{\alpha}M$ for $\alpha>0$ can be no easier to compute than $\per M$; however, when $k$ is a negative integer, we have cancellation of all terms indexed by $\pi$ with $\#\pi>-k$, yielding
\[\per_{-k}M=\sum_{\pi\in\partitionsnk}k^{\downarrow\#\pi}\det(M\cdot\pi).\]
When $M$ is positive semi-definite, all terms on the right-hand side above are non-negative and it seems reasonable that computation might be easier than the general $\alpha$ case.

Even if $M$ is not positive semi-definite, a reasonable approximation might be possible because the sum is over $\partitionsnk$, which is asymptotically of much smaller order than $\partitionsn$.  Writing $B(n,k):=\{\pi\in\partitionsn:\#\pi=k\}$, the $(n,k)$th Bell number, and $B(n):=\sum_k B(n,k)$, we see that, for fixed $k\in\mathbb{N}$,
\[B(n,\leq k):=\sum_{j=1}^kB(n,k)\sim k^n/k!\quad\mbox{as }n\rightarrow\infty\]
and $B(n)\sim n^{-1/2}[\lambda(n)]^{n+1/2}e^{\lambda(n)-n-1}$ as $n\rightarrow\infty$, where $\lambda(n)=e^{W(n)}$, for $W(n)$ the {\em Lambert W function}.  Identity \eqref{eq:per det beta} is also consistent with our knowledge that computation of the determinant is much easier than the permanent; the case $\beta=-1$ is the degenerate sum over one term in \eqref{eq:per det beta}.  Our investigation prompts the following conjecture.
\begin{conj}
Computation of the $\alpha$-permanent is $\#$P-complete, except for $\alpha\in\{-1,\ldots,-K\}$, for some $K\in\mathbb{N}\cup\{\infty\}$.  The complexity for $\alpha\in\{-2,-3,\ldots\}$ lies somewhere between P and $\#$P and is monotonically non-decreasing as $\alpha$ decreases.
\end{conj}
We leave open the possibility that $K=1$ above, in which case our conjecture is that the $\alpha$-permanent is $\#$P-complete except for $\alpha=-1$.  The crux of our conjecture, however, is that the $\alpha$-permanent must be (at least) $\#$P-complete unless $\alpha$ is a negative integer.

\section{Appendix}
\subsection{Matrices for numerical approximation}\label{section:matrices numerical}
Below is the matrix $X_1$ used in approximations in Table \ref{table:1}.  The matrix is symmetric about the diagonal and we include only the upper triangular part of $X_1$.
\[X_1=\begin{pmatrix}
4.42 & 3.13 & 3.14 & 3.45 & 4.01 & 3.85 & 3.39 & 2.70\\
 & 2.70 & 1.99 & 2.44 & 3.07 & 2.83 & 2.27 & 1.84\\
 & &3.52 & 2.26 & 2.73 & 2.43 & 2.85 & 2.36\\
& & & 3.57 & 3.01 & 3.17 & 2.93 & 1.90\\
& & & &4.12 & 3.69 & 3.01 & 2.03\\
& & & & & 3.91 & 3.03 & 2.08\\
& & & & & & 3.27 & 2.22\\
& & & & & & & 2.33
\end{pmatrix}
\]
\subsection{Generalized rencontres numbers}\label{section:rencontres}
Here, we list tables of the generalized rencontres numbers from section \ref{section:special matrices} for $n\leq 10$.  Blank entries correspond to 0.
\[\begin{tabular}{c|ccc}
$(2,k,l)$ & l=0 & 1 & 2\\\hline
k=1 & 1 & & \\
2 & 0 & 2& 1
\end{tabular}\quad\quad\begin{tabular}{c|cccc}
$(3,k,l)$ & 0 & 1 & 2 & 3\\\hline
1 & 2 & & &\\
2 & 0 & 3 & & \\
3 & 0 & 0 & 0 &1
\end{tabular}
\quad\quad
\begin{tabular}{ c|ccccc }
$(4,k,l)$ & 0& 1 & 2 & 3 & 4\\\hline
1 & 6 & & & & \\
2 & 3 & 8 & & & \\
3 & 0 & 0 & 6 & & \\
4 & 0 & 0 & 0 & 0 & 1
\end{tabular}\]
\[
\begin{tabular}{ c|cccccc }
$(5,k,l)$ & 0& 1 & 2 & 3 & 4 &5 \\\hline
1 & 24 & & & & &\\
2 & 20  & 30 & & & &\\
3 & 0 & 15 & 20 & & &\\
4 & 0 & 0 & 0 &10 & &\\
5 & 0 & 0 & 0 & 0 & 0 & 1
\end{tabular}\quad\quad
\begin{tabular}{ c|cccccccc }
$(6,k,l)$ & 0& 1 & 2 & 3 & 4 &5 &6\\\hline
1 & 120 & & & & &&\\
2 & 130  & 144 & & & &&\\
3 & 15 & 120 & 90 & & &&\\
4 & 0 & 0 & 45 &40 & &&\\
5 & 0 & 0 & 0 & 0 & 15 & &\\
6 & 0 & 0 & 0 & 0 & 0 & 0 & 1
\end{tabular}
\]
\[\begin{tabular}{ c|ccccccccc }
$(7,k,l)$ & 0& 1 & 2 & 3 & 4 &5 &6 &7\\\hline
1 & 720& & & & & &&\\
2 & 924  & 840& & & & &&\\
3 & 210 & 910 & 504 & & & &&\\
4 & 0 & 105 & 420 & 210 & & &&\\
5 & 0 & 0 & 0 & 105 & 70& & &\\
6 & 0 & 0 & 0 & 0 & 0 & 21 & &\\
7 & 0 & 0 & 0 & 0 & 0 & 0 & 0 & 1
\end{tabular}
\]
\[\begin{tabular}{ c|cccccccccc }
$(8,k,l)$ & 0& 1 & 2 & 3 & 4 &5 &6 &7&8\\\hline
1 & 5040&& & & & & &&\\
2 & 7308  & 5760&& & & & &&\\
3 & 2380 & 7392& 3360& & & & &&\\
4 & 105 & 1680 & 3640 & 1344 & & &&\\
5 & 0 & 0 & 420 & 1120 & 420& & &\\
6 & 0 & 0 & 0 & 0 & 210 & 112 & &\\
7 & 0 & 0 & 0 & 0 & 0 & 0 & 28 & &\\
8 & 0 & 0 & 0 & 0 & 0 & 0 & 0 &0 & 1
\end{tabular}
\]
\[\begin{tabular}{ c|ccccccccccc }
$(9,k,l)$ & 0& 1 & 2 & 3 & 4 &5 &6 &7&8&9\\\hline
1 & 40320&& & & & & &&\\
2 & 64224  & 45360&& & & & &&\\
3 & 26432 & 65772& 25920& & & & &&\\
4 & 2520 & 21420 & 33264 & 10080& & & &&\\
5 & 0 & 945 & 7560 & 10920 & 3024&& & &\\
6 & 0 & 0 & 0 & 1260 & 2520& 756&& & &\\
7 & 0 & 0 & 0 & 0 & 0 & 378 & 168 & &&\\
8 & 0 & 0 & 0 & 0 & 0 & 0 & 0 &36 & &\\
9 &  0 & 0 & 0 & 0 & 0 & 0 & 0 &0 & 0&1
\end{tabular}
\]
\[\begin{tabular}{ c|cccccccccccc }
$(10,k,l)$ & 0& 1 & 2 & 3 & 4 &5 &6 &7&8&9&10\\\hline
1 & 362880&  && & & & & &&\\
2 & 623376  & 403200&&& & & & &&\\
3 & 303660 & 642240& 226800& & & & &&\\
4 & 44100 & 264320 & 328860 & 86400&& & & &&\\
5 & 945 & 25200 & 107100 & 110880 & 25200&&& & &\\
6 & 0 & 0 &4725 & 25200 & 27300& 6048&&& & &\\
7 & 0 & 0 & 0 & 0 & 3150 & 5040 & 1260& & &&\\
8 & 0 & 0 & 0 & 0 & 0 & 0 & 630 & 240& & &\\
9 &  0 & 0 & 0 & 0 & 0 & 0 & 0 &0 & 45&&\\
10& 0 & 0 & 0 & 0 & 0 & 0 & 0 &0 & 0&0&1
\end{tabular}
\]

\begin{acknowledgment}
I especially thank Peter McCullagh for several insights while preparing this work. 
\end{acknowledgment}

\bibliography{crane-refs}

\begin{thebibliography}{10}

\bibitem{Crane2013b}
H.~Crane.
\newblock Permanental partition measures and {M}arkovian {G}ibbs partitions.
\newblock {\em Unpublished manuscript in preparation}, 2013.

\bibitem{Ewens1972}
W.~J. Ewens.
\newblock The sampling theory of selectively neutral alleles.
\newblock {\em Theoret. Population Biology}, 3:87--112, 1972.

\bibitem{FultonHarris}
W.~Fulton and J.~Harris.
\newblock {\em Representation Theory: A First Course}.
\newblock Graduate {T}exts in {M}athematics/{R}eadings in {M}athematics.
  Springer, 1991.

\bibitem{HoughDetSurvey}
J.~B. Hough, M.~Krishnapur, Y.~Peres, and B.~Vir\'ag.
\newblock Determinantal {P}rocesses and {I}ndependence.
\newblock {\em Probability Surveys}, 3:206--229, 2006.

\bibitem{JerrumSinclairVigoda}
M.~Jerrum, A.~Sinclair, and E.~Vigoda.
\newblock A polynomial-time approximation algorithm for the permanent of a
  matrix with non-negative entries.
\newblock {\em Journal of the ACM}, pages 671--697, 2004.

\bibitem{KouMcCullagh2009}
S.~C. Kou and P.~McCullagh.
\newblock Approximating the {$\alpha$}-permanent.
\newblock {\em Biometrika}, 96(3):635--644, 2009.

\bibitem{LinialSW1998}
N.~Linial, A.~Samordnitsky, and A.~Wigderson.
\newblock A deterministic strongly polynomial algorithm for matrix scaling and
  approximate permanents.
\newblock In {\em Proc.\ 30th ACM Symp. on Theory of Computing}, New York,
  1998. ACM.

\bibitem{MarcusMinc1961}
M.~Marcus and M.~Henryk.
\newblock On the relation between the determinant and the permanent.
\newblock {\em Illinois J. {M}ath}, 5(3):376--381, 1961.

\bibitem{McCullagh2012}
P.~McCullagh.
\newblock An asymptotic approximation for the permanent of a doubly stochastic
  matrix.
\newblock {\em Journal of Statistical Computation and Simulation}, pages 1--11,
  2012.

\bibitem{McCullagh2005}
P.~McCullagh and J.~M{\o}ller.
\newblock The permanental process.
\newblock {\em Adv. in Appl. Probab.}, 38(4):873--888, 2006.

\bibitem{MincPermanent}
H.~Minc.
\newblock {\em Permanents}.
\newblock Addison-Wesley, Reading, MA, 1978.

\bibitem{Pitman2005}
J.~Pitman.
\newblock {\em Combinatorial stochastic processes}, volume 1875 of {\em Lecture
  Notes in Mathematics}.
\newblock Springer-Verlag, Berlin, 2006.
\newblock Lectures from the 32nd Summer School on Probability Theory held in
  Saint-Flour, July 7--24, 2002, With a foreword by Jean Picard.

\bibitem{PressIS2007}
W.~Press, S.~Teukolsky, V.~W.T., and B.~Flannery.
\newblock {\em Numerical {R}ecipes: {T}he {A}rt of {S}cientific {C}omputing,
  3rd Ed.}
\newblock Cambridge University Press, New York, 2007.

\bibitem{RubakMollerMcC2010}
E.~Rubak, J.~M\o~ller, and P.~McCullagh.
\newblock Statistical inference for a class of multivariate negative binomial
  distributions.
\newblock Technical Report R-2010-10. Department of Mathematical Sciences,
  Aalborg University, Aalborg, Denmark, 2010.

\bibitem{Sos2000}
A.~Soshnikov.
\newblock Determinantal random point fields.
\newblock {\em {R}ussian {M}ath. {S}urveys}, 55(5):923--975, 2000.

\bibitem{StanleyI}
R.~Stanley.
\newblock {\em Enumerative {C}ombinatorics, Volume 1}, volume~49 of {\em
  Cambridge Studies in Advanced Mathematics}.
\newblock Cambridge University Press, New York, 2012.
\newblock Second Edition.

\bibitem{ValiantPermanent}
L.~G. Valiant.
\newblock The complexity of computing the permanent.
\newblock {\em Theoretical Computer Science}, 8:189--201, 1979.

\bibitem{VereJones1988}
D.~Vere-Jones.
\newblock A generalization of permanents and determinants.
\newblock {\em Linear Alg. Appl.}, 111:119--124, 1988.

\end{thebibliography}
\bibliographystyle{abbrv}
\end{document}